\documentclass[12pt, reqno]{amsart}
\usepackage{amsmath, amsthm, amscd, amsfonts, amssymb, graphicx, color}
\usepackage[bookmarksnumbered, colorlinks, plainpages]{hyperref}

\hypersetup{colorlinks=true,linkcolor=red, anchorcolor=green, citecolor=cyan, urlcolor=red, filecolor=magenta, pdftoolbar=true}
\newtheorem{theorem}{Theorem}[section]
\newtheorem{lemma}[theorem]{Lemma}

\theoremstyle{definition}

\theoremstyle{remark}

\numberwithin{equation}{section}
\numberwithin{equation}{section}

\begin{document}

\setcounter{page}{1}

\begin{center}
{\Large \textbf{\ Approximation properties on $q$-Sz\'{a}%
sz-Mirakjan-Kantrovich Stancu type operators via Dunkl generalization }}

\bigskip

\textbf{M. Mursaleen} and \textbf{Md. Nasiruzzaman}

Department of\ Mathematics, Aligarh Muslim University, Aligarh--202002, India%
\\[0pt]

mursaleenm@gmail.com; nasir3489@gmail.com \\[0pt]

\bigskip

\bigskip

\textbf{Abstract}
\end{center}

\parindent=8mm {\footnotesize {\ This paper is devoted to study the approximation
properties and rate of approximation of the Sz\'{a}sz-Mirakjan-Kantrovich-Stancu type
polynomials generated by the Dunkl generalization of the exponential function with
respect to $q$-calculus. We present approximation properties with the help of
well-known Korovkin's theorem and determine the rate of convergence in terms of
classical modulus of continuity, the class of Lipschitz functions, Peetre's K-functional, and the
second-order modulus of continuity. Moreover, we obtain the approximation
results for Bivariate case for these operators.}}

\bigskip



\section{Introduction and preliminaries}
S.N Bernstein \cite{sbbl1} introduced the very well known Bernstein Classical operators in 1912 as

\begin{equation}
B_{n}(f;x)=\sum_{k=0}^{n}\binom{n}{k}x^{k}(1-x)^{n-k}f\left( \frac{k}{n}%
\right) ,~~~~~~~x\in \lbrack 0,1].  \label{s1}
\end{equation}%
for $n\in \mathbb{N}$ and $f\in C[0,1]$.\newline

In 1950, for $x \geq 0$, Sz\'{a}sz \cite{sbbl4} introduced the operators
\begin{equation}  \label{s2}
S_n(f;x)=e^{-nx}\sum_{k=0}^\infty \frac{(nx)^k}{k!} f\left(\frac{k}{n}%
\right),~~~~~~~f \in C[0,\infty).
\end{equation}

The first $q$-analogue of the well-known Bernstein polynomials was introduced
by Lupa\c{s} \cite{sbbl2}, by applying the idea of $q$-integers and investigated its approximating
and shape-preserving properties. In 1997 Phillips \cite{sbbl3} considered another $q$-analogue of
the classical Bernstein polynomials. Later on, many authors introduced $q$%
-generalizations of various operators and investigated several approximation
properties \cite{smah1,smah2,smur1,smur3,smn2} etc. Recently, making use of $%
(p,q)$-calculus, the first $(p,q)$-analogue of Bernstein operators was given
in \cite{sn1}. There is a generalization in $(p,q)$-calculus.

The $q$-integer $[n]_{q}$, the $q$-factorial $[n]_{q}!$ and the $q$-binomial
coefficient are defined by (see \cite{sbbl5})
\begin{align*}
\lbrack n]_{q}& :=\left\{
\begin{array}{ll}
\frac{1-q^{n}}{1-q}, & \hbox{if~}q\in \mathbb{R}^{+}\setminus \{1\} \\
n, & \hbox{if~}q=1,%
\end{array}%
\right. \mbox{for $n\in \mathbb{N} $~and~$[0]_q=0$}, \\
\lbrack n]_{q}!& :=\left\{
\begin{array}{ll}
\lbrack n]_{q}[n-1]_{q}\cdots \lbrack 1]_{q}, & \hbox{$n\geq 1$,} \\
1, & \hbox{$n=0$,}%
\end{array}%
\right. \\
\left[
\begin{array}{c}
n \\
k%
\end{array}%
\right] _{q}& :=\frac{[n]_{q}!}{[k]_{q}![n-k]_{q}!},
\end{align*}%
respectively.

\noindent The $q$-analogue of $(1+x)^{n}$ is defined by
\begin{equation*}
(1+x)_{q}^{n}:=\left\{
\begin{array}{ll}
(1+x)(1+qx)\cdots (1+q^{n-1}x) & \quad n=1,2,3,\cdots \\
1 & \quad n=0.%
\end{array}%
\right.
\end{equation*}

\noindent The Gauss binomial formula is given by
\begin{equation*}
(x+a)_{q}^{n}=\sum\limits_{k=0}^{n}\left[
\begin{array}{c}
n \\
k%
\end{array}%
\right] _{q}q^{k(k-1)/2}a^{k}x^{n-k}.
\end{equation*}

In 1930 the first Bernstein-Kantorovich operators \cite{kantiv} in 1930 and in 1950 Sz\'{a}asz
operators \cite{sbbl4} were defined. In the recent years Sucu \cite{sbbl9} define a
Dunkl analogue of Sz\'{a}sz operators via a generalization of the exponential function given
by \cite{sbbl10} as follows:

\begin{equation}
S_{n}^*(f;x):= \frac{1}{e_\mu(nx)}\sum_{k=0}^\infty \frac{(nx)^k}{%
\gamma_\mu(k)} f \left(\frac{k+2\mu\theta_k}{n}\right),
\end{equation}
where $x \geq 0,~~~f \in C[0,\infty),\mu \geq 0,~~~n \in \mathbb{N}$\newline
and
\begin{equation*}
e_\mu(x)= \sum_{n=0}^\infty \frac{x^n}{\gamma_\mu(n)}.
\end{equation*}
Here
\begin{equation*}
\gamma_\mu(2k)= \frac{2^{2k}k!\Gamma\left(k+\mu+\frac{1}{2}\right)}{%
\Gamma\left(\mu+\frac{1}{2}\right)},
\end{equation*}
and
\begin{equation*}
\gamma_\mu(2k+1)= \frac{2^{2k+1}k!\Gamma\left(k+\mu+\frac{3}{2}\right)}{%
\Gamma\left(\mu+\frac{1}{2}\right)}.
\end{equation*}

Cheikh et al. \cite{sbbl11} stated the $q$-Dunkl classical $q$-Hermite type
polynomials and gave definitions of $q$-Dunkl analogues of exponential
functions and recursion relations for $\mu >-\frac{1}{2}$ and $0<q<1$.
\begin{equation}
e_{\mu ,q}(x)=\sum_{n=0}^{\infty }\frac{x^{n}}{\gamma _{\mu ,q}(n)},~~~x\in
\lbrack 0,\infty )  \label{sr1}
\end{equation}%
\begin{equation}
E_{\mu ,q}(x)=\sum_{n=0}^{\infty }\frac{q^{\frac{n(n-1)}{2}}x^{n}}{\gamma
_{\mu ,q}(n)},~~~x\in \lbrack 0,\infty )  \label{sr2}
\end{equation}%
\begin{equation}
\gamma _{\mu ,q}(n+1)=\left( \frac{1-q^{2\mu \theta _{n+1}+n+1}}{1-q}\right)
\gamma _{\mu ,q}(n),~~~~~n\in \mathbb{N},  \label{sr3}
\end{equation}

\begin{equation*}
\theta _{n}=%
\begin{cases}
0 & \quad \text{if }n\in 2\mathbb{N}, \\
1 & \quad \text{if }n\in 2\mathbb{N}+1.%
\end{cases}%
\end{equation*}%
Recently \.{I}\c{c}\={o}z and \c{C}ekim gave a Dunkl generalization of
Kantrovich integral generalization of Sz\'{a}sz operators in \cite{ckz} and
a Dunkl generalization of Sz\'{a}sz operators via $q$-calculus in \cite%
{sbbl13}.

This paper is dedicated to construct Kantrovich Stancu type Sz\'{a}sz-Mirakjan operators generated by
the Dunkl generalization of the exponential function. We study approximation properties
of the operators via a universal Korovkin’s type theorem and a weighted Korovkin’s
type theorem and the rate of convergence of the operators for functions belonging to the
Lipschitz class are presented.




\section{auxiliary results}

We define a Dunkl generalization of Sz\'{a}sz-Mirakjan-Kantrovich Stancu type
operators via $q$-calculs: \newline

For any $x\in [0,\infty), ~~~n \in \mathbb{N}, ~~0<q< 1,$ and $\mu>\frac{1}{2%
}$, we define

\begin{equation}
\tilde{K_{n,q}}(f;x)=\frac{[n]_{q}}{E_{\mu ,q}([n]_{q}x)}\sum_{k=0}^{\infty }%
\frac{([n]_{q}x)^{k}}{\gamma _{\mu ,q}(k)}q^{\frac{k(k-1)}{2}}\int_{\frac{%
[k+2\mu \theta _{k}]_{q}}{q^{k-2}[n]_{q}}}^{\frac{[k+1+2\mu \theta
_{k}]_{q}-1}{q^{k-1}[n]_{q}}+\frac{1}{[n]_{q}}}f\left(\frac{nt+\alpha}{n+\beta}\right)d_{q}t,  \label{nss1}
\end{equation}%
where $f$ is continuous and nondecreasing function on $[0,\infty )$. When we take $\alpha=\beta=0$ in the operator $\tilde{K_{n,q}}(f;x)$ \eqref{nss1}, then we obtain the Kantorovich-type generalization of
Sz\'{a}sz operators \cite{sno6}.

\begin{lemma}
\label{nlm1} Let $K_{n,q}^{\ast }(.~;~.)$ be the operators given by %
\eqref{nss1}. Then we have:

\begin{enumerate}
\item \label{nlm11} $\Tilde{K_{n,q}}(1;x)=1$

\item \label{nlm12} $\Tilde{K_{n,q}}(t;x)=\frac{n}{(n+\beta)}\left(\frac{\alpha}{n}+\frac{1}{ [2]_q[n]_q}\right)+\frac{2nq}{[2]_q(n+\beta)}x$

\item \label{nlm13}$\frac{n^2}{(n+\beta)^2}\left(\frac{\alpha^2}{n^2}+\frac{2 \alpha}{[2]_q n [n]_q}+
\frac{ 1}{ [3]_q [n]_q^2 (n+\beta)}\right)
+ \frac{ q n^2 }{(n+\beta)^2}\left(\frac{3 q^{2 \mu +1}[1+2\mu]_q}{[3]_q[n]_q} +\frac{3}{[3]_q[n]_q}+ \frac{4 \alpha  }{[2]_q n } \right)x+
\frac{3q n^2}{[3]_q(n+\beta)^2}x^2  \leq \Tilde{K_{n,q}}%
(t^2;x) \leq \newline
\frac{n^2}{(n+\beta)^2}\left(\frac{\alpha^2}{n^2}+\frac{2 \alpha}{[2]_q n [n]_q}+
\frac{ 1}{ [3]_q [n]_q^2 (n+\beta)}\right)
+ \frac{n^2}{(n+\beta)^2}\left(\frac{3 [1+2\mu]_q}{[3]_q[n]_q} +\frac{3q}{[3]_q[n]_q}+ \frac{4 q \alpha  }{[2]_q n } \right)x+
\frac{3n^2}{[3]_q(n+\beta)^2}x^2 $.
\end{enumerate}
\end{lemma}

\begin{lemma} \label{cvc}
\label{nlm2} Let the operators $\Tilde{K_{n,q}}(.~;~.)$ be given by %
\eqref{nss1}. Then we have,

\begin{enumerate}
\item \label{nlm21} $\Tilde{K_{n,q}}(t-1;x)=\frac{n}{(n+\beta)}\left(\frac{\alpha}{n}+\frac{1}{ [2]_q[n]_q}- \frac{n+\beta}{n}\right)+\frac{2nq}{[2]_q(n+\beta)}x$

\item \label{nlm22} $\Tilde{K_{n,q}}(t-x;x)=\frac{n}{(n+\beta)}\left(\frac{\alpha}{n}+\frac{1}{ [2]_q[n]_q}\right)+\left(\frac{2nq}{[2]_q(n+\beta)}-1\right)x$

\item \label{nlm23} $ \Tilde{K_{n,q}}((t-x)^2;x)\leq  \frac{n^2}{(n+\beta)^2}\left(\frac{\alpha^2}{n^2}+\frac{2 \alpha}{[2]_q n [n]_q}+
\frac{ 1}{ [3]_q [n]_q^2 (n+\beta)}\right)\newline
+ \left\{\frac{n^2}{(n+\beta)^2}\left(\frac{3 [1+2\mu]_q}{[3]_q[n]_q} +\frac{3q}{[3]_q[n]_q}+ \frac{4 q \alpha  }{[2]_q n } \right)
-  \frac{2n}{(n+\beta)}\left(\frac{\alpha}{n}+\frac{1}{ [2]_q[n]_q}\right) \right\}x
+ \left(\frac{3n^2}{[3]_q(n+\beta)^2}- \frac{2nq}{[2]_q(n+\beta)}+1 \right) x^2$
\end{enumerate}
\end{lemma}




\section{Korovkin type approximation properties}

In this section, we obtain the Korovkin type and weighted Korovkin type
approximation properties by the operators defined in \eqref{nss1}.

Korovkin-type theorems furnish simple and useful tools for ascertaining
whether a given sequence of positive linear operators, acting on some
function space, is an approximation process.

\begin{theorem}
\label{t1}(\cite{cnf,fnc}) Let $L_n:C[a,b]\to C[a,b]$ be a sequence of
positive linear operators. If $\lim\limits_{n\to\infty}L_n t^j=t^j,~ j=0,1,2$%
, uniformly on $[a,b]$, then $\lim\limits_{n\to\infty}L_nf=f$ uniformly on $%
[a,b]$ for every $f\in C[a,b]$.
\end{theorem}

Let $C_{B}(\mathbb{R^{+}})$ be the set of all bounded and continuous
functions on $\mathbb{R^{+}}=[0,\infty )$, which is linear normed space with
\begin{equation*}
\parallel f\parallel _{C_{B}}=\sup_{x\geq 0}\mid f(x)\mid .
\end{equation*}%
Also let
\begin{equation*}
H:=\{f:x\in \lbrack 0,\infty ),\frac{f(x)}{1+x^{2}}~~~\mbox{is}~~~%
\mbox{convergent}~~~\mbox{as}~~~x\rightarrow \infty \}.
\end{equation*}


\parindent=8mmIn order to obtain the convergence results for the operators $%
\Tilde{K_{n,q}}(.,.)$, we take $q=q_{n}$ where $q_{n}\in (0,1)$ and
satisfying,
\begin{equation}
\lim_{n}q_{n}\rightarrow 1,~~~~~~\lim_{n}q_{n}^{n}\rightarrow a
\label{nnas5}
\end{equation}

\begin{theorem}
\label{nth1} Let $q=q_{n}$ satisfy \eqref{nnas5}, for $0<q_{n}<1$ and if $%
\Tilde{K_{n,q_{n}}}(.~;~.)$ be the operators given by \eqref{nss1}. Then for
any function $f\in C[0,\infty )\cap H$,
\begin{equation*}
\Tilde{K_{n,q_{n}}}(f;x)=f(x)
\end{equation*}%
is uniformly on each compact subset of $[0,\infty )$.
\end{theorem}

\begin{proof}
The proof is based on the well known Korovkin's theorem regarding the
convergence of a sequence of linear positive operators, so it is enough to
prove the conditions
\begin{equation*}
\Tilde{K_{n,q_{n}}}((t^{j};x)=x^{j},~~~j=0,1,2,~~~\{\mbox{as}~n\rightarrow
\infty \}
\end{equation*}%
uniformly on $[0,1]$.\newline
Clearly from \eqref{nnas5} and $\frac{1}{[n]_{q_{n}}}\rightarrow
0~~(n\rightarrow \infty ),$ we have

\begin{equation*}
\lim_{n\rightarrow \infty }{K}_{n,q_{n}}^{\ast
}(t;x)=x,~~~\lim_{n\rightarrow \infty }{K}_{n,q_{n}}^{\ast }(t^{2};x)=x^{2}.
\end{equation*}%
This complete the proof.
\end{proof}

The weighted Korovkin-type theorems were proved by Gadzhiev \cite{gad}. A
real function $\rho = 1 + x^2$ is called a weight function if it is
continuous on $\mathbb{R}$ and
\begin{equation*}
\lim_{\mid x \mid \to \infty}\rho(x)=\infty,~~~\rho(x)\geq 1~~\mbox{for}~~%
\mbox{all}~~x \in \mathbb{R}
\end{equation*}

We recall the weighted spaces of the functions on $\mathbb{R}^{+}$, which
are defined as follows:
\begin{eqnarray*}
P_{\rho }(\mathbb{R}^{+}) &=&\left\{ f:\mid f(x)\mid \leq M_{f}\rho
(x)\right\} , \\
Q_{\rho }(\mathbb{R}^{+}) &=&\left\{ f:f\in P_{\rho }(\mathbb{R}^{+})\cap
C[0,\infty )\right\} , \\
Q_{\rho }^{k}(\mathbb{R}^{+}) &=&\left\{ f:f\in Q_{\rho }(\mathbb{R}^{+})~~~%
\mbox{and}~~~\lim_{x\rightarrow \infty }\frac{f(x)}{\rho (x)}=k(k~~~\mbox{is}%
~~~\mbox{a}~~~\mbox{constant})\right\} ,
\end{eqnarray*}%
where $\rho (x)=1+x^{2}$ is a weight function and $M_{f}$ is a constant
depending only on $f$. Note that $Q_{\rho }(\mathbb{R}^{+})$ is a normed
space with the norm $\parallel f\parallel _{\rho }=\sup_{x\geq 0}\frac{\mid
f(x)\mid }{\rho (x)}$.

\begin{theorem}
\label{nth2} Let $q=q_{n}$ satisfy \eqref{nnas5} for $0<q_{n}<1$ and let $%
\Tilde{K_{n,q_{n}}}(.~;~.)$ be the operators given by \eqref{nss1}. Then for
any function $f\in Q_{\rho }^{k}(\mathbb{R}^{+})$ we have
\begin{equation*}
\lim_{n\rightarrow \infty }\parallel \Tilde{K_{n,q_{n}}}(f;x)-f\parallel
_{\rho }=0.
\end{equation*}
\end{theorem}

\begin{proof}
From Lemma \ref{nlm1}, the first condition of \eqref{nlm11} is fulfilled for
$\tau =0$. Now for $\tau =1,2$ it is easy to see from (\ref{nlm12}), (\ref%
{nlm13}) of Lemma \ref{nlm1} by using \eqref{nnas5} that \newline
\begin{equation*}
\parallel \Tilde{K_{n,q_{n}}}\left( t)^{\tau };x\right) -x^{\tau }\parallel
_{\rho }=0.
\end{equation*}%
This complete the proof.
\end{proof}


\section{ Rate of Convergence}

Here we calculate the rate of convergence of operators \eqref{nss1} by means
of modulus of continuity and Lipschitz type maximal functions.

Let $f\in C[0,\infty ]$. The modulus of continuity of $f$ denoted by $\omega
(f,\delta )$ gives the maximum oscillation of $f$ in any interval of length
not exceeding $\delta >0$ and it is given by the relation
\begin{equation}
\omega (f,\delta )=\sup_{\mid y-x\mid \leq \delta }\mid f(y)-f(x)\mid
,~~~x,y\in \lbrack 0,\infty ).  \label{nson1}
\end{equation}%
It is known that $\lim_{\delta \rightarrow 0+}\omega (f,\delta )=0$ for $%
f\in C[0,\infty )$ and for any $\delta >0$ one has
\begin{equation}
\mid f(y)-f(x)\mid \leq \left( \frac{\mid y-x\mid }{\delta }+1\right) \omega
(f,\delta ).  \label{nson2}
\end{equation}

\begin{theorem}\label{tbb1}
Let $\Tilde{K_{n,q}}(.~;~.)$ be the operators defined by \eqref{nss1}. Then
for $f\in \tilde{C}[0,\infty),~~x\geq 0, ~~0<q< 1$ we have
\begin{equation*}
\mid \Tilde{K_{n,q}}(f;x)-f(x)\mid \leq \left( 1+\sqrt{\lambda _{n}(x)}\right) \omega\left(f;\frac{1}{\sqrt{[n]_{q}}}\right),
\end{equation*}%
where $\lambda _{n}(x)=\Tilde{K_{n,q}}\left( (t-x)^{2};x\right) $ is defined in (\ref{nlm23}) of Lemma \ref{cvc} and $\tilde{C}[0,\infty)$ is the space of uniformly continuous functions
on $\mathbb{R}^+$, furthermore, $\omega(f,\delta)$ is the modulus of continuity of the
function $f \in \tilde{C}[0,\infty)$ defined in \eqref{nson1}.
\end{theorem}

\begin{proof}
We prove it by using \eqref{nson1},\eqref{nson2} and Cauchy-Schwarz
inequality.\newline
$\mid \Tilde{K_{n,q}}(f;x)-f(x)\mid $
\begin{eqnarray*}
&\leq &\frac{[n]_{q}}{E_{\mu ,q}([n]_{q}x)}\sum_{k=0}^{\infty }\frac{%
([n]_{q}x)^{k}}{\gamma _{\mu ,q}(k)}q^{\frac{k(k-1)}{2}}\int_{\frac{[k+2\mu
\theta _{k}]_{q}}{q^{k-2}[n]_{q}}}^{\frac{[k+1+2\mu \theta _{k}]_{q}-1}{%
q^{k-1}[n]_{q}}+\frac{1}{[n]_{q}}}\mid f(t)-f(x)\mid d_{q}(t) \\
&\leq &\frac{[n]_{q}}{E_{\mu ,q}([n]_{q}x)}\sum_{k=0}^{\infty }\frac{%
([n]_{q}x)^{k}}{\gamma _{\mu ,q}(k)}q^{\frac{k(k-1)}{2}}\int_{\frac{[k+2\mu
\theta _{k}]_{q}}{q^{k-2}[n]_{q}}}^{\frac{[k+1+2\mu \theta _{k}]_{q}-1}{%
q^{k-1}[n]_{q}}+\frac{1}{[n]_{q}}}\left( 1+\frac{1}{\delta }\mid t-x\mid
\right) d_{q}(t)\omega (f;\delta ) \\
&=&\left\{ 1+\frac{1}{\delta }\left( \frac{[n]_{q}}{E_{\mu ,q}([n]_{q}x)}%
\sum_{k=0}^{\infty }\frac{([n]_{q}x)^{k}}{\gamma _{\mu ,q}(k)}q^{\frac{k(k-1)%
}{2}}\int_{\frac{[k+2\mu \theta _{k}]_{q}}{q^{k-2}[n]_{q}}}^{\frac{[k+1+2\mu
\theta _{k}]_{q}-1}{q^{k-1}[n]_{q}}+\frac{1}{[n]_{q}}}\mid t-x\mid
d_{q}(t)\right) \right\} \omega (f;\delta ) \\
&\leq &\left\{ 1+\frac{1}{\delta }\left( \frac{[n]_{q}}{E_{\mu ,q}([n]_{q}x)}%
\sum_{k=0}^{\infty }\frac{([n]_{q}x)^{k}}{\gamma _{\mu ,q}(k)}q^{\frac{k(k-1)%
}{2}}\int_{\frac{[k+2\mu \theta _{k}]_{q}}{q^{k-2}[n]_{q}}}^{\frac{[k+1+2\mu
\theta _{k}]_{q}-1}{q^{k-1}[n]_{q}}+\frac{1}{[n]_{q}}}(t-x)^{2}d_{q}(t)%
\right) ^{\frac{1}{2}}\left( \Tilde{K_{n,q}}(1;x)\right) ^{\frac{1}{2}%
}\right\} \\
&&\times \omega (f;\delta ) \\
&=&\left\{ 1+\frac{1}{\delta }\left( \Tilde{K_{n,q}}(t-x)^{2};x\right) ^{%
\frac{1}{2}}\right\} \omega (f;\delta ) \\
&&
\end{eqnarray*}%
if we choose $\delta =\delta _{n}=\sqrt{\frac{1}{[n]_{q}}}$, then we get our
result.
\end{proof}

Now we give the rate of convergence of the operators ${K}_{n,q}^*(f;x) $
defined in \eqref{nss1} in terms of the elements of the usual Lipschitz
class $Lip_{M}(\nu )$.

Let $f\in C[0,\infty )$, $M>0$ and $0<\nu \leq 1$. The class $Lip_{M}(\nu )$
is defined as
\begin{equation}
Lip_{M}(\nu )=\left\{ f:\mid f(\zeta _{1})-f(\zeta _{2})\mid \leq M\mid
\zeta _{1}-\zeta _{2}\mid ^{\nu }~~~(\zeta _{1},\zeta _{2}\in \lbrack
0,\infty ))\right\} .  \label{nn1}
\end{equation}

\begin{theorem}
\label{nsn1} Let $\Tilde{K_{n,q}}(.~;~.)$ be the operators defined in %
\eqref{nss1}. Then for each $f\in Lip_{M}(\nu ),~(M>0,~~~0<\nu \leq 1)$
satisfying \eqref{nnn1} we have
\begin{equation*}
\mid \Tilde{K_{n,q}}(f;x)-f(x)\mid \leq M\left( \lambda _{n}(x)\right) ^{%
\frac{\nu }{2}}
\end{equation*}%
where $\lambda _{n}(x)=\Tilde{K_{n,q}}\left( (t-x)^{2};x\right) $ defined in Theorem \ref{tbb1}.
\end{theorem}

\begin{proof}
We prove it by using \eqref{nnn1} and H\"{o}lder's inequality.
\begin{eqnarray*}
\mid \Tilde{K_{n,q}}(f;x)-f(x)\mid &\leq &\mid \Tilde{K_{n,q}}%
(f(t)-f(x);x)\mid \\
&\leq &\Tilde{K_{n,q}}\left( \mid f(t)-f(x)\mid ;x\right) \\
&\leq &\mid M\Tilde{K_{n,q}}\left( \mid t-x\mid ^{\nu };x\right) .
\end{eqnarray*}%
Therefore\newline

$\mid \Tilde{K_{n,q}}(f;x)-f(x) \mid$
\begin{eqnarray*}
&\leq & M \frac{[n]_q}{E_{\mu,q}([n]_qx)}\sum_{k=0}^\infty \frac{([n]_qx)^{k}%
}{\gamma_{\mu,q}(k)}q^{\frac{k(k-1)}{2}} \int_{\frac{[k+2\mu \theta_k]_q}{%
q^{k-2}[n]_q}}^{\frac{[k+1+2\mu \theta_k]_q-1}{q^{k-1}[n]_q}+\frac{1}{[n]_q}%
}\mid t-x \mid^\nu d_q(t) \\
& \leq & M \frac{[n]_q}{E_{\mu,q}([n]_qx)}\sum_{k=0}^\infty \left(\frac{%
([n]_qx)^{k}q^{\frac{k(k-1)}{2}}}{\gamma_{\mu,q}(k)}\right)^{\frac{2-\nu}{2}}
\\
& \times & \left(\frac{([n]_qx)^{k}q^{\frac{k(k-1)}{2}}}{\gamma_{\mu,q}(k)}%
\right)^{\frac{\nu}{2}} \int_{\frac{[k+2\mu \theta_k]_q}{q^{k-2}[n]_q}}^{%
\frac{[k+1+2\mu \theta_k]_q-1}{q^{k-1}[n]_q}+\frac{1}{[n]_q}}\mid t-x
\mid^\nu d_q(t) \\
& \leq & M \left(\frac{[n]_q}{\left(E_{\mu,q}([n]_qx)\right)}%
\sum_{k=0}^\infty \frac{([n]_qx)^{k}q^{\frac{k(k-1)}{2}}}{\gamma_{\mu,q}(k)}%
\int_{\frac{[k+2\mu \theta_k]_q}{q^{k-2}[n]_q}}^{\frac{[k+1+2\mu
\theta_k]_q-1}{q^{k-1}[n]_q}+\frac{1}{[n]_q}} d_q(t)\right)^{\frac{2-\nu}{2}}
\\
& \times & \left(\frac{[n]_q}{\left(E_{\mu,q}([n]_qx)\right)}%
\sum_{k=0}^\infty \frac{([n]_qx)^{k}q^{\frac{k(k-1)}{2}}}{\gamma_{\mu,q}(k)}
\int_{\frac{[k+2\mu \theta_k]_q}{q^{k-2}[n]_q}}^{\frac{[k+1+2\mu
\theta_k]_q-1}{q^{k-1}[n]_q}+\frac{1}{[n]_q}}\mid t-x \mid^2 d_q(t) \right)^{%
\frac{\nu}{2}} \\
& = & M \left(\Tilde{K_{n,q}}(t-x)^2;x\right)^{\frac{\nu}{2}}.
\end{eqnarray*}
Which complete the proof.
\end{proof}

Let $C_{B}[0,\infty )$ denote the space of all bounded and continuous
functions on $\mathbb{R}^{+}=[0,\infty )$ and
\begin{equation}
C_{B}^{2}(\mathbb{R}^{+})=\{g\in C_{B}(\mathbb{R}^{+}):g^{\prime },g^{\prime
\prime }\in C_{B}(\mathbb{R}^{+})\},  \label{nt2}
\end{equation}%
with the norm
\begin{equation}
\parallel g\parallel _{C_{B}^{2}(\mathbb{R}^{+})}=\parallel g\parallel
_{C_{B}(\mathbb{R}^{+})}+\parallel g^{\prime }\parallel _{C_{B}(\mathbb{R}%
^{+})}+\parallel g^{\prime \prime }\parallel _{C_{B}(\mathbb{R}^{+})},
\label{nt1}
\end{equation}%
also
\begin{equation}
\parallel g\parallel _{C_{B}(\mathbb{R}^{+})}=\sup_{x\in \mathbb{R}^{+}}\mid
g(x)\mid .  \label{nt3}
\end{equation}

\begin{theorem}
\label{nsn2} Let $\Tilde{K_{n,q}}(.~;~.)$ be the operators defined in %
\eqref{nss1}. Then for any $g\in C_{B}^{2}(\mathbb{R}^{+})$ we have
\begin{equation*}
\mid \Tilde{K_{n,q}}(f;x)-f(x)\mid \leq \left\{ \frac{n}{(n+\beta)}\left(\frac{\alpha}{n}+\frac{1}{ [2]_q[n]_q}\right)+\left(\frac{2nq}{[2]_q(n+\beta)}-1\right)x+\frac{\lambda _{n}(x)}{2}\right\}
\parallel g\parallel _{C_{B}^{2}(\mathbb{R}^{+})}
\end{equation*}%
where $\lambda _{n}(x)$ is given in Theorem \ref{tbb1}.
\end{theorem}

\begin{proof}
Let $g\in C_{B}^{2}(\mathbb{R}^{+})$. Then by using the generalized mean
value theorem in the Taylor series expansion we have
\begin{equation*}
g(t)=g(x)+g^{\prime }(x)(t-x)+g^{\prime \prime }(\psi )\frac{(t-x)^{2}}{2}%
,~~~\psi \in (x,t).
\end{equation*}%
By linearity, we have
\begin{equation*}
\Tilde{K_{n,q}}(g,x)-g(x)=g^{\prime }(x)\Tilde{K_{n,q}}\left( (t-x);x\right)
+\frac{g^{\prime \prime }(\psi )}{2}\Tilde{K_{n,q}}\left( (t-x)^{2};x\right)
,
\end{equation*}%
which implies that\newline
$\mid \Tilde{K_{n,q}}(g;x)-g(x)\mid $
\begin{eqnarray*}
&\leq &\left(   \Tilde{K_{n,q}}\left( (t-x);x\right) \right) \parallel g^{\prime }\parallel _{C_{B}(\mathbb{R}^{+})}
+\left(\Tilde{K_{n,q}}\left( (t-x)^2;x\right) \right) \frac{\parallel g^{\prime \prime
}\parallel _{C_{B}(\mathbb{R}^{+})}}{2}.
\end{eqnarray*}%
From \eqref{nt1} we have ~~~$\parallel g^{\prime }\parallel _{C_{B}[0,\infty
)}\leq \parallel g\parallel _{C_{B}^{2}[0,\infty )}$.\newline
$\mid \Tilde{K_{n,q}}(g;x)-g(x)\mid $
\begin{eqnarray*}
&\leq &\left( \Tilde{K_{n,q}}\left( (t-x);x\right)\right) \parallel g\parallel _{C_{B}^{2}(\mathbb{R}^{+})}
+\left( \Tilde{K_{n,q}}\left( (t-x)^{2};x\right)\right) \frac{\parallel g\parallel
_{C_{B}^{2}(\mathbb{R}^{+})}}{2}.
\end{eqnarray*}%
This completes the proof from Lemma \ref{nlm2}.
\end{proof}

The Peetre's $K$-functional is defined by
\begin{equation}  \label{nzr1}
K_{2}(f,\delta )=\inf_{C_B^2(\mathbb{R}^+)} \left\{ \left( \parallel
f-g\parallel_{C_B(\mathbb{R}^+)} +\delta \parallel g^{\prime \prime
}\parallel_{C_B^2(\mathbb{R}^+)} \right):g\in \mathcal{W}^{2}\right\} ,
\end{equation}%
where
\begin{equation}  \label{nzr2}
\mathcal{W}^{2}=\left\{ g\in C_B(\mathbb{R}^+):g^{\prime },g^{\prime \prime
}\in C_B(\mathbb{R}^+)\right\} .
\end{equation}%
Then there exits a positive constant $C>0$ such that $K_{2}(f,\delta )\leq C
\omega _{2}(f,\delta ^{\frac{1}{2}}),~~\delta >0$, where the second order
modulus of continuity is given by
\begin{equation}  \label{nzr3}
\omega _{2}(f,\delta ^{\frac{1}{2}})=\sup_{0<h<\delta ^{\frac{1}{2}%
}}\sup_{x\in \mathbb{R}^+}\mid f(x+2h)-2f(x+h)+f(x)\mid .
\end{equation}

\begin{theorem}
\label{nsn3} Let $\Tilde{K_{n,q}}(.~;~.)$ be the operators defined in %
\eqref{nss1} and $C_{B}[0,\infty )$ be the space of all bounded and
continuous functions on $\mathbb{R}^{+}$. Then for $x\in \mathbb{R}%
^{+},~f\in C_{B}(\mathbb{R}^{+})$ we have\newline
$\mid \Tilde{K_{n,q}}(f;x)-f(x)\mid $\newline

$\leq 2M \bigg{\{} \omega_2  \left( \sqrt{\frac{\lambda _{n}(x)%
}{4}
+ \frac{n}{2(n+\beta)}\left(\frac{\alpha}{n}+\frac{1}{ [2]_q[n]_q}\right)+\frac{1}{2}\left(\frac{2nq}{[2]_q(n+\beta)}-1\right)x}\right)$\newline
$+\min\left(1,\frac{\lambda _{n}(x)}{4}
+ \frac{n}{2(n+\beta)}\left(\frac{\alpha}{n}+\frac{1}{ [2]_q[n]_q}\right)+\frac{1}{2}\left(\frac{2nq}{[2]_q(n+\beta)}-1\right)x \right)\bigg{\}}$,\newline
where $M$ is a positive constant, $\lambda _{n}(x)$ is given in Theorem \ref{tbb1} and $\omega _{2}(f;\delta )$ is the second order modulus of
continuity of the function $f$ defined in \eqref{nzr3}.
\end{theorem}

\begin{proof}
We prove this by using Theorem \eqref{nsn2}
\begin{eqnarray*}
\mid \Tilde{K_{n,q}}(f;x)-f(x)\mid &\leq &\mid \Tilde{K_{n,q}}(f-g;x)\mid
+\mid \Tilde{K_{n,q}}(g;x)-g(x)\mid +\mid f(x)-g(x)\mid \\
&\leq &2\parallel f-g\parallel _{C_{B}(\mathbb{R}^{+})}+\frac{\lambda _{n}(x)%
}{2}\parallel g\parallel _{C_{B}^{2}(\mathbb{R}^{+})} \\
&+&\left\{ \frac{n}{(n+\beta)}\left(\frac{\alpha}{n}+\frac{1}{ [2]_q[n]_q}\right)+\left(\frac{2nq}{[2]_q(n+\beta)}-1\right)x\right\}
 \parallel g\parallel _{C_{B}(\mathbb{R}^{+})}
\end{eqnarray*}%
From \eqref{nt1} clearly we have ~~~$\parallel g\parallel _{C_{B}[0,\infty
)}\leq \parallel g\parallel _{C_{B}^{2}[0,\infty )}$.\newline
Therefore,\newline
$\mid \Tilde{K_{n,q}}(f;x)-f(x)\mid $
\begin{equation*}
\leq 2\parallel f-g\parallel _{C_{B}(\mathbb{R}^{+})}+\left\{\frac{\lambda _{n}(x)%
}{2}
+ \frac{n}{(n+\beta)}\left(\frac{\alpha}{n}+\frac{1}{ [2]_q[n]_q}\right)+\left(\frac{2nq}{[2]_q(n+\beta)}-1\right)x\right\}\parallel g\parallel _{C_{B}^{2}(\mathbb{R}^{+})}
\end{equation*}%
where $\lambda _{n}(x)$ is given in Theorem \ref{tbb1}.\newline

By taking infimum over all $g\in C_{B}^{2}(\mathbb{R}^{+})$ and by using %
\eqref{nzr1}, we get
\begin{equation*}
\mid \Tilde{K_{n,q}}(f;x)-f(x)\mid \leq 2K_{2}\left\{ f;\frac{\lambda _{n}(x)%
}{4}
+ \frac{n}{2(n+\beta)}\left(\frac{\alpha}{n}+\frac{1}{ [2]_q[n]_q}\right)+\frac{1}{2}\left(\frac{2nq}{[2]_q(n+\beta)}-1\right)x\right\}
\end{equation*}%
Now for an absolute constant $C>0$ in \cite{scc1} we use the relation
\begin{equation*}
K_{2}(f;\delta )\leq C\{\omega _{2}(f;\sqrt{\delta })+\min (1,\delta
)\parallel f\parallel \}.
\end{equation*}%
This completes the proof of the theorem.
\end{proof}


\section{Construction of Bivariate Operators}

In this section, we construct a bivariate extension of the operators %
\eqref{nss1}.

Let $\mathbb{R}^2_+ = [0,\infty)\times [0,\infty),~~f: C(\mathbb{R}^2_+ )\to
\mathbb{R}$ and $0<q_{n_1},q_{n_2}<1$. We define the bivariate extension of
the Dunkl $q$-parametric Sz\'{a}sz-Mirakjan operators \eqref{nss1} as
follows:\newline

\begin{equation*}
K_{n_1,n_2}^*(f;q_{n_1},q_{n_2};x,y)=\frac{1}{E_{%
\mu_1,q_{n_1}}([n_1]_{q_{n_1}}x)}\frac{1}{E_{\mu_2,q_{n_2}}([n_2]_{q_{n_2}}y)%
} \sum_{k_1=0}^\infty\sum_{k_2=0}^\infty \frac{([n_1]_{q_{n_1}}x)^{k_1}}{%
\gamma_{\mu_1,q_{n_1}}(k_1)} \frac{([n_2]_{q_{n_2}}y)^{k_2}}{%
\gamma_{\mu_2,q_{n_2}}(k_2)}
\end{equation*}

\begin{equation}
\times q_{n_{1}}^{\frac{k_{1}(k_{1}-1)}{2}}q_{n_{2}}^{\frac{k_{2}(k_{2}-1)}{2%
}}\int_{~~R}\int f\left(\frac{n_1x+\alpha_1}{n_1+\beta_1},\frac{n_2y+\alpha_2}{n_2+\beta_2}\right)d_{q}yd_{q}x  \label{nst1}
\end{equation}%
where\newline
$R=\bigg{[}\frac{[k_{1}+2\mu _{1}\theta _{k_{1}}]_{q_{n_{1}}}}{%
q_{n_{1}}^{k_{1}-2}[n_{1}]_{q_{n_{1}}}},\frac{[k_{1}+1+2\mu _{1}\theta
_{k_{1}}]_{q_{n_{1}}-1}}{q_{n_{1}}^{k_{1}-1}[n_{1}]_{q_{n_{1}}}}+\frac{1}{%
[n_{1}]_{q_{n_{1}}}}\bigg{]}$

$\times \bigg{[}\frac{[k_{2}+2\mu _{2}\theta _{k_{2}}]_{q_{n_{2}}}}{%
q_{n_{2}}^{k_{2}-2}[n_{2}]_{q_{n_{2}}}},\frac{[k_{2}+1+2\mu _{2}\theta
_{k_{2}}]_{q_{n_{2}}-1}}{q_{n_{2}}^{k_{2}-1}[n_{2}]_{q_{n_{2}}}}+\frac{1}{%
[n_{2}]_{q_{n_{2}}}}\bigg{]}$

$=\bigg{\{}(x,y)\in \mathbb{R}^{2}~~\mbox{such}~~\mbox{that}~~%
\frac{[k_{1}+2\mu _{1}\theta _{k_{1}}]_{q_{n_{1}}}}{%
q_{n_{1}}^{k_{1}-2}[n_{1}]_{q_{n_{1}}}}\leq x\leq \frac{\lbrack k_{1}+1+2\mu
_{1}\theta _{k_{1}}]_{q_{n_{1}}-1}}{q_{n_{1}}^{k_{1}-1}[n_{1}]_{q_{n_{1}}}}+%
\frac{1}{[n_{1}]_{q_{n_{1}}}}~~\mbox{and}~~$

$\frac{[k_{2}+2\mu _{2}\theta _{k_{2}}]_{q_{n_{2}}}}{%
q_{n_{2}}^{k_{2}-2}[n_{2}]_{q_{n_{2}}}}\leq y\leq \frac{\lbrack k_{2}+1+2\mu
_{2}\theta _{k_{2}}]_{q_{n_{2}}-1}}{q_{n_{2}}^{k_{1}-1}[n_{2}]_{q_{n_{2}}}}+%
\frac{1}{[n_{2}]_{q_{n_{2}}}}\bigg{\}}$,\newline
and

\begin{equation*}
E_{\mu_1,q_{n_1}}([n_1]_{q_{n_1}}x)=\sum_{k_1=0}^\infty\frac{%
([n_1]_{q_{n_1}}x)^{k_1}}{\gamma_{\mu_1,q_{n_1}}(k_1)}q_{n_1}^{\frac{%
k_1(k_1-1)}{2}},~~~~~ E_{\mu_2,q_{n_2}}([n_2]_{q_{n_2}}y)=\sum_{k_2=0}^\infty%
\frac{([n_2]_{q_{n_2}}y)^{k_2}}{\gamma_{\mu_2,q_{n_2}}(k_2)}q_{n_2}^{\frac{%
k_2(k_2-1)}{2}}.
\end{equation*}

\begin{lemma}
Let $e_{i,j}:\mathbb{R}_{+}^{2}\rightarrow \lbrack 0,\infty )$ be such that $%
e_{i,j}=u^{i}v^{j},~i,j=0,1,2$ be the two dimensional test functions. Then
the $q$-bivariate operators defined in \eqref{nst1} satisfy the following:

\begin{enumerate}
\item \label{nl1} $K_{n_1,n_2}^*(e_{0,0};q_{n_1},q_{n_2};x,y)=1$

\item \label{nl2} $K_{n_1,n_2}^*(e_{1,0};q_{n_1},q_{n_2};x,y)=\frac{n_1}{(n_1+\beta_1)}\left(\frac{\alpha_1}{n_1}+\frac{1}{ [2]_{q_{n_1}}[n_1]_{q_{n_1}}}\right)+\frac{2n_1q_{n_1}}{[2]_{q_{n_1}}(n_1+\beta_1)}x$

\item \label{nl3} $K_{n_1,n_2}^*(e_{0,1};q_{n_1},q_{n_2};x,y)=\frac{n_2}{(n_2+\beta_2)}\left(\frac{\alpha_2}{n_2}+\frac{1}{ [2]_{q_{n_2}}[n_2]_{q_{n_2}}}\right)+\frac{2n_2q_{n_2}}{[2]_{q_{n_2}}(n_2+\beta_2)}y$

\item \label{nl4} $K_{n_1,n_2}^*(e_{2,0};q_{n_1},q_{n_2};x,y)\leq \frac{n_1^2}{(n_1+\beta_1)^2}\left(\frac{\alpha_1^2}{n_1^2}+\frac{2 \alpha_1}{[2]_{q_{n_1}} n_1 [n_1]_{q_{n_1}}}+
\frac{ 1}{ [3]_{q_{n_1}} [n_1]_{q_{n_1}}^2 (n_1+\beta_1)}\right)
+ \frac{n_1^2}{(n_1+\beta_1)^2}\left(\frac{3 [1+2\mu_1]_{q_{n_1}}}{[3]_{q_{n_1}}[n_1]_{q_{n_1}}} +\frac{3q_{n_1}}{[3]_{q_{n_1}}[n_1]_{q_{n_1}}}+ \frac{4 q_{n_1} \alpha_1  }{[2]_{q_{n_1}} n_1 } \right)x+
\frac{3n_1^2}{[3]_{q_{n_1}}(n_1+\beta_1)^2}x^2  $

\item \label{nl5}$K_{n_1,n_2}^*(e_{0,2};q_{n_1},q_{n_2};x,y)\leq \frac{n_2^2}{(n_2+\beta_2)^2}\left(\frac{\alpha_2^2}{n_2^2}+\frac{2 \alpha_2}{[2]_{q_{n_2}} n_2 [n_2]_{q_{n_2}}}+
\frac{ 1}{ [3]_{q_{n_2}} [n_2]_{q_{n_2}}^2 (n_2+\beta_2)}\right)
+ \frac{n_2^2}{(n_2+\beta_2)^2}\left(\frac{3 [1+2\mu_2]_{q_{n_2}}}{[3]_{q_{n_2}}[n_2]_{q_{n_2}}} +\frac{3q_{n_2}}{[3]_{q_{n_2}}[n_2]_{q_{n_2}}}+ \frac{4 q_{n_2} \alpha_2  }{[2]_{q_{n_2}} n_2 } \right)y+
\frac{3n_2^2}{[3]_{q_{n_2}}(n_2+\beta_2)^2}y^2  $.
\end{enumerate}
\end{lemma}

\begin{lemma}\label{lml}
The $q$-bivariate operators defined in \eqref{nst1} satisfy the following:

\begin{enumerate}
\item \label{nll1} $K_{n_1,n_2}^*(e_{1,0}-x;q_{n_1},q_{n_2};x,y)=\frac{n_1}{(n_1+\beta_1)}\left(\frac{\alpha_1}{n_1}+\frac{1}{ [2]_{q_{n_1}}[n_1]_{q_{n_1}}}\right)+\left(\frac{2n_1q_{n_1}}{[2]_{q_{n_1}}(n_1+\beta_1)}-1\right)x
$

\item \label{nll2} $K_{n_1,n_2}^*(e_{0,1}-y;q_{n_1},q_{n_2};x,y)=\frac{n_2}{(n_2+\beta_2)}\left(\frac{\alpha_2}{n_2}+\frac{1}{ [2]_{q_{n_2}}[n_2]_{q_{n_2}}}\right)+\left(\frac{2n_2q_{n_2}}{[2]_{q_{n_2}}(n_2+\beta_2)}-1\right)y$

\item \label{nll3} $K_{n_{1},n_{2}}^{\ast }\left( (e_{1,0}-x\right)
^{2};q_{n_{1}},q_{n_{2}};x,y)$\newline
$\leq \frac{n_1^2}{(n_1+\beta_1)^2}\left(\frac{\alpha_1^2}{n_1^2}+\frac{2 \alpha_1}{[2]_{q_{n_1}} n_1 [n_1]_{q_{n_1}}}+
\frac{ 1}{ [3]_{q_{n_1}} [n_1]_{q_{n_1}}^2 (n_1+\beta_1)}\right)\newline
+ \left\{\frac{n_1^2}{(n_1+\beta_1)^2}\left(\frac{3 [1+2\mu_1]_{q_{n_1}}}{[3]_{q_{n_1}}[n_1]_{q_{n_1}}} +\frac{3q_{n_1}}{[3]_{q_{n_1}}[n_1]_{q_{n_1}}}+ \frac{4 q_{n_1} \alpha_1  }{[2]_{q_{n_1}} n_1 } \right)
-  \frac{2n_1}{(n_1+\beta_1)}\left(\frac{\alpha_1}{n_1}+\frac{1}{ [2]_{q_{n_1}}[n_1]_{q_{n_1}}}\right) \right\}x
+ \left(\frac{3n_1^2}{[3]_{q_{n_1}}(n_1+\beta_1)^2}- \frac{2n_1q_{n_1}}{[2]_{q_{n_1}}(n_1+\beta_1)}+1 \right) x^2 $

\item $K_{n_{1},n_{2}}^{\ast }\left( (e_{0,1}-y\right)
^{2};q_{n_{1}},q_{n_{2}};x,y)$\newline
$\leq \frac{n_2^2}{(n_2+\beta_2)^2}\left(\frac{\alpha_2^2}{n_2^2}+\frac{2 \alpha_2}{[2]_{q_{n_2}} n_2 [n_2]_{q_{n_2}}}+
\frac{ 1}{ [3]_{q_{n_2}} [n_2]_{q_{n_2}}^2 (n_2+\beta_2)}\right)\newline
+ \left\{\frac{n_2^2}{(n_2+\beta_2)^2}\left(\frac{3 [1+2\mu_2]_{q_{n_2}}}{[3]_{q_{n_2}}[n_2]_{q_{n_2}}} +\frac{3q_{n_2}}{[3]_{q_{n_2}}[n_2]_{q_{n_2}}}+ \frac{4 q_{n_2} \alpha_2  }{[2]_{q_{n_2}} n_2 } \right)
-  \frac{2n_2}{(n_2+\beta_2)}\left(\frac{\alpha_2}{n_2}+\frac{1}{ [2]_{q_{n_2}}[n_2]_{q_{n_2}}}\right) \right\}y
+ \left(\frac{3n_2^2}{[3]_{q_{n_2}}(n_2+\beta_2)^2}- \frac{2n_2q_{n_2}}{[2]_{q_{n_2}}(n_2+\beta_2)}+1 \right) y^2$.\\
\end{enumerate}
\end{lemma}

\parindent=8mmIn order to obtain the convergence results for the operators $%
K_{n_{1},n_{2}}^{\ast }(f;q_{n_{1}},q_{n_{2}};x,y)$, we take $%
q=q_{n_{1}},~~q_{n_{2}}$ where $q_{n_{1}},q_{n_{2}}\in (0,1)$ and satisfy
\begin{equation}
\lim_{n_{1},n_{2}}q_{n_{1}},q_{n_{2}}\rightarrow 1.  \label{nnasi5}
\end{equation}

The modulus of continuity for bivariate case is defined as follows:\newline

For $f\in H_{\omega }(\mathbb{R}_{+}^{2})$\newline
$\widetilde{\omega }(f;\delta _{1},\delta _{2})$
\begin{equation}  \label{nsnn1}
=\sup_{u,x\geq 0}\left\{ \bigg{|}f(u,v)-f(x,y)\bigg{|};~~\bigg{|}u-x\bigg{|}%
\leq \delta _{1},\bigg{|}v-y\bigg{|}\leq \delta _{2},~~(u,v)\in \mathbb{R}%
_{+}^{2},~~(x,y)\in \mathbb{R}_{+}^{2}\right\}.
\end{equation}

where $H_{\omega }(\mathbb{R}^{+})$ is the space of all real-valued
continuous functions $f$.\newline
Then for all $f\in H_{\omega }(\mathbb{R}_{+})$ $\widetilde{\omega }%
(f;\delta _{1},\delta _{2})$ we have the following conditions:

\begin{enumerate}
\item[$($i$)$] $\lim_{\delta_1,\delta_2 \to 0}\widetilde{\omega}(f;
\delta_1,\delta_2)\to 0$

\item[$($ii$)$] $\mid f(u,v)-f(x,y) \mid \leq \widetilde{\omega}(f;
\delta_1,\delta_2) \left(\frac{\mid u-x\mid}{\delta_1}+1 \right)\left(\frac{%
\mid v-y\mid}{\delta_2}+1 \right).$
\end{enumerate}

\begin{theorem}
\label{ntn1} Let $q=q_{n_{1}},q_{n_{2}}$ satisfy \eqref{nnasi5} and $%
K_{n_{1},n_{2}}^{\ast }(f;q_{n_{1}},q_{n_{2}},x,y)$ be the operators defined
by \eqref{nst1}. Then for any function $f\in \tilde{C}\left( [0,\infty
)\times \lbrack 0,\infty )\right) $, for $(x,y)\in \lbrack 0,\infty
),~0<q_{n_{1}},q_{n_{2}}<1$, we have \newline
$\mid K_{n_{1},n_{2}}^{\ast }(f;q_{n_{1}},q_{n_{2}},x,y)-f(x,y)\mid $
\begin{eqnarray*}
&\leq & \omega \left( f;\frac{1}{\sqrt{[n_{1}]_{q_{n_{1}}}}},\frac{1}{\sqrt{%
[n_{2}]_{q_{n_{2}}}}}\right)\lambda_{n_1}(x)\lambda_{n_2}(y),
\end{eqnarray*}%
where $\lambda_{n_1}(x)=K_{n_1,n_2}^*\left((e_{1,0}-x)^2;q_{n_1},q_{n_2};x,y%
\right),~~~~~
\lambda_{n_2}(y)=K_{n_1,n_2}^*\left((e_{0,1}-y)^2;q_{n_1},q_{n_2};x,y\right)$ are defined in Lemma \ref{lml} and $\tilde{C}[0,\infty )$ is the space of uniformly continuous functions
on $\mathbb{R}^{+}$. Moreover, $\widetilde{\omega }(f,\delta _{n_{1}},\delta
_{n_{2}})$ is the modulus of continuity of the function $f\in \tilde{C}%
\left( [0,\infty )\times \lbrack 0,\infty )\right) $ defined in \eqref{nsnn1}.%
\end{theorem}

\begin{proof}
We can prove easily by using the Cauchy-Schwarz inequality and choosing $%
\delta _{1}=\delta _{n_{1}}=\sqrt{\frac{1}{[n_{1}]_{q_{n_{1}}}}}$ and $%
\delta _{2}=\delta _{n_{2}}=\sqrt{\frac{1}{[n_{2}]_{q_{n_{2}}}}}$ . So we
omit the details.
\end{proof}

Now we give the rate of convergence of the operators ${K}%
_{n_1,n_2}^*(f;q_{n_1},q_{n_2};x,y) $ defined in \eqref{nst1} in terms of
the elements of the usual Lipschitz class $Lip_{M}(\nu_1,\nu_2 )$.

Let $f\in C([0,\infty )\times \lbrack 0,\infty ))$, $M>0$ and $0<\nu
_{1},\nu _{2}\leq 1$. The class $Lip_{M}(\nu _{1},\nu _{2})$ is defined by
\begin{equation}
Lip_{M}(\nu _{1},\nu _{2})=\left\{ f:\mid f(u,v)-f(x,y)\mid \leq M\mid
u-x\mid ^{\nu _{1}}\mid v-y\mid ^{\nu _{2}}~~~(u,v)~~~\mbox{and}~~~(x,y)\in
\lbrack 0,\infty )\right\} .  \label{nnn1}
\end{equation}

\begin{theorem}
\label{nsy1} Let $K_{n_1,n_2}^*(f;q_{n_1},q_{n_2};x,y)$ be the operator
defined in \eqref{nst1}. Then for each $f\in Lip_{M}(\nu_1,\nu_2
),~~(M>0,~~~0<\nu_1,\nu_2 \leq 1)$ satisfying \eqref{nnn1} we have
\begin{equation*}
\mid K_{n_1,n_2}^*(f;q_{n_1},q_{n_2};x,y)-f(x,y)\mid \leq M
\left(\lambda_{n_1}(x)\right)^{\frac{\nu_1}{2}}
\left(\lambda_{n_2}(y)\right)^{\frac{\nu_2}{2}},
\end{equation*}
where $\lambda_{n_1}(x)$ and $\lambda_{n_2}(y)$ are defined in Theorem \ref{ntn1}.
\end{theorem}

\begin{proof}
We can prove easily it by using \eqref{nnn1} and H\"{o}lder inequality. So
we omit the details.
\end{proof}


\textbf{Acknowledgement.} Second author (MN) acknowledges the financial sup-
port of University Grants Commission (Govt. of Ind.) for awarding BSR (Basic
Scientific Research) Fellowship.

\end{document}